\documentclass[a4]{article}
\usepackage{amsmath,amssymb,amsthm,mathtools}
\usepackage{comment}
\usepackage[utf8]{inputenc}
\usepackage[margin=1.5in]{geometry}
\usepackage{hyperref}
\hypersetup{colorlinks=true,linkcolor=blue,citecolor=blue,urlcolor=blue}
\usepackage{geometry}
\usepackage{amsmath, amssymb, amsthm, amsfonts}
\usepackage{mathtools}
\usepackage{physics}
\usepackage{tikz-cd}
\usepackage{enumitem}
\usepackage[
backend=biber,
style=alphabetic,
sorting=nyt
]{biblatex}
\theoremstyle{definition}
\newtheorem{definition}{Definition}
\newtheorem{question}{Question}
\newtheorem{example}{Example}
\newtheorem{remark}{Remark}

\theoremstyle{plain}
\newtheorem{theorem}{Theorem}

\newtheorem{lemma}{Lemma}
\newtheorem{proposition}{Proposition}

\newcommand{\D}{\mathbb{D}}
\newcommand{\Dbar}{\overline{\mathbb{D}}}
\newcommand{\Dbarstar}{\overline{\mathbb{D}}_*}
\newcommand{\Z}{\mathbb{Z}}
\newcommand{\R}{\mathbb{R}}

\newcommand{\A}{\mathcal{A}}

\newcommand{\psu}{\text{PSU}(1,1)}
\newcommand{\psun}{\text{PSU}^n(1,1)}
\newcommand{\slR}{\mathfrak{sl}_2(\mathbb{R})}

\newcommand{\An}{\mathcal{A}_n}
\newtheorem{maintheorem}{Theorem}

\title{Lie groupoid integration of singular isometries of  the Poincaré disk}
\author{Rea Dalipi}
\date{}

\begin{document}
\maketitle

\begin{abstract}
For every $n \geq 1$ there is a distinguished $\mathfrak{sl}_2(\mathbb{R})$ action on the Poincaré disk, arising as the infinitesimal isometries of a hyperbolic metric with conical singularity of order $n-1$ at the origin. For $n=1$ this is the standard infinitesimal Möbius action, and for $n>1$ these vector fields have singularities at the origin and are incomplete, preventing integration to a global Lie group action.
However, they naturally define an action Lie algebroid $\mathcal{A}_n=\mathfrak{sl}_2(\mathbb{R})\ltimes \Dbarstar $ over the punctured disk. We construct an explicit Lie groupoid $\mathcal{G}_n$ integrating $\mathcal{A}_n$ and compare it to the \v Severa-Weinstein groupoid. Although $\mathcal{G}_n$ is not an action groupoid, its restriction to the boundary recovers an $n$-fold Möbius action on the boundary circle.
\end{abstract}

\tableofcontents

\section{Introduction and main results}
The Virasoro group and its coadjoint orbits have long played a prominent role in mathematics and physics, notably in the study of two-dimensional conformal field theories \cite{Alekseev:1988ce,Rai:1989js}. More recently, they have reappeared in the correspondence between Jackiw--Teitelboim gravity on the disk and the Schwarzian quantum mechanics on the boundary circle \cite{saad2019jtgravitymatrixintegral}. A key result in this direction is the link between Virasoro coadjoint orbits and moduli spaces of hyperbolic metrics, explored in detail in \cite{AM24}. The authors proved that there is a symplectomorphism between the moduli space of hyperbolic metrics on the disk and Virasoro's first exceptional orbit:
\begin{equation}\label{virasoro}
   \text{Hyp}(\D)/\text{Diff}^+(\D, \partial \D) \cong \mathcal{O}_{T=\frac{1}{2}}.
\end{equation}
Naturally, one is led to ask whether such a symplectomorphism can be generalized to the higher exceptional orbits corresponding to $T=\frac{n^2}{2}$. It turns out \cite{alexeevdalipishatashvili-exceptionalorbits} that in order to reproduce the higher orbits one should replace the smooth hyperbolic metrics on the left-hand side of  \eqref{virasoro} by hyperbolic metrics with prescribed singularities.  Studying this generalized setting, we encountered the issue of determining the global isometries of such singular metrics, and answering this question constitutes the goal of this manuscript.

Consider the Poincaré disk $\D=\{ z \in \mathbb{C}, |z|<1 \}$, with the action of $\psu$ by Möbius transformations, that is for any $g=\begin{pmatrix}
    a & b\\ \bar{b} &\bar{a}
\end{pmatrix},$ $a,b \in \mathbb{C}$:

\begin{equation*}
    g\cdot z=\frac{az+b}{\bar b z+\bar a}.
\end{equation*}
The $\psu$-action extends to the closed Poincaré disk $\Dbar=\{ z \in \mathbb{C}, |z|\leq 1 \}$, and it preserves the standard hyperbolic metric:
\begin{equation}\label{gstd}
g_{\Dbar}=\frac{4|dz|^2}{(1-|z|^2)^2}.
\end{equation}
We are interested in hyperbolic metrics on $\Dbar\setminus \{p\}$, which have a prescribed singularity at $p$. As we will see in more detail in Section \ref{section2}, such singularities arise from holomorphic maps which locally look like  $p_n: z \mapsto z^n$. 
Pulling back the smooth metric in \eqref{gstd} via this branched cover yields 
\begin{equation*}
g_n=p_n^*g_{\Dbar}
=
\frac{4n^2 |z|^{2(n-1)}}{(1-|z|^{2n})^2}|dz|^2 ,
\end{equation*}
which is a smooth metric away from the origin, but becomes singular at $z=0$. Thus $g_n$ defines a metric on the punctured disk $\Dbarstar=\{ z \in \mathbb{C},\ 0< |z|\leq 1 \}$. Note that for $n=1$, we recover the smooth disk with its smooth standard metric \eqref{gstd}.

Given a group action of
$G$ on a surface $\Sigma$, covering theory guarantees that this action lifts to an action of $G^n$ on the $n$-fold cover $\Sigma^n$,  under suitable topological conditions.
 The map $p_n$ is not a covering map of the disk, but a branched one. Therefore, the $\psu$ action does not necessarily lift to a $\psun$ action.  The obstruction becomes clearer as we attempt to explicitly write down the lifted action. One can show (see Section \ref{section2}) that for a fixed $n \in \mathbb{N}$, the multivalued map
\begin{equation*}
g^{(n)}\cdot z
=
\left(\frac{az^n+b}{\bar b z^n+\bar a}\right)^{1/n}, \qquad \abs{z}=1
\end{equation*}
maps the boundary circle into itself and that it has $n$ single-valued branches, uniquely determining a globally defined $\psun$ action on $S^1$. This formula fails to extend to the interior of the disk; there the $n$-th root can only be defined locally. This local action yields the following
$\mathfrak{sl}_2(\mathbb{R})$ vector fields on $\Dbarstar$:
\begin{equation}\label{vf0}
    \begin{aligned}
    E &= \frac{1}{n}(z^{1-n} - z^{1+n})\partial_{z} + (\overline{z}^{1-n} - \overline{z}^{1+n})\partial_{\overline{z}} ,\\
    F &= \frac{i}{n}(z^{1-n} + z^{1+n})\partial_{z} - (\overline{z}^{1-n} +  \overline{z}^{1+n})\partial_{\overline{z}}, \\
    H &= \frac{2i}{n}(z\partial_{z} - \overline{z}\partial_{\overline{z}}).
\end{aligned}
\end{equation}
For $n=1$ (the smooth case), the $\mathfrak{sl}_2(\mathbb{R})$ action integrates to a global $\psu$ action. On the other hand, for $n > 1$, these vector fields are singular at $z=0$, and they do not integrate to a global $\psun$ action on $\Dbarstar$ due to incompleteness. 
 To remedy this issue, we turn towards higher symmetry structures, namely \textit{Lie algebroids and groupoids}. In classical Lie theory, the relationship between Lie groups and Lie algebras is governed by the three integrability theorems of Lie.\footnote{Following the terminology of \cite{Mackenzie_2005}.} 
 
These theorems admit higher generalizations to Lie algebroids and Lie groupoids.

A Lie algebra  action of $\mathfrak g$  on a manifold $M$ defines an \textit{action Lie algebroid}, as the trivial vector bundle
\begin{equation*}
\mathfrak g\ltimes M\longrightarrow M
\end{equation*}
with anchor
\begin{equation*}
\rho:\mathfrak g\ltimes M\longrightarrow TM,
\qquad
(\xi,x)\longmapsto \xi_M(x).
\end{equation*}
Then, given the $\mathfrak{sl}_2(\mathbb{R})$ action of the vector fields in
\eqref{vf0}, we can define the action Lie algebroid
\begin{equation*}
\mathcal{A}_n\coloneq\mathfrak{sl}_2(\mathbb R)\ltimes \Dbarstar
\end{equation*}
with anchor map the vector fields $E,F,H$.

An \textit{integration} of a Lie algebroid $\mathcal{A}\to M$ is a Lie groupoid
\begin{equation*}
\mathcal G\rightrightarrows M
\end{equation*}
whose Lie algebroid is isomorphic to $\mathcal{A}$. Whenever $\mathcal{A}$ is integrable, there is a unique (up to isomorphisms) source-simply connected integration groupoid, namely the \v{S}evera--Weinstein groupoid \(\mathcal G(\mathcal{A})\). In our case, the following theorem guarantees integrability:
\begin{theorem}\cite{DAZORD199777} \label{actionlie}
    Every action Lie algebroid is integrable.
\end{theorem}
While Dazord's work establishes existence and uniqueness of the integrating groupoid, it is often difficult to describe it explicitly.
Our goal is therefore to construct a more concrete model of $\mathcal{A}_n$.

We consider the set of pairs 
\begin{equation*}
X_n
=
\left\{
(u,z)\ \middle|\ 
u:[0,1]\to \psu,\ u(0)=e,\ z\in \Dbarstar,\ 
u(t)\cdot z^n\neq 0
\right\}.
\end{equation*} 
Note that (we will prove it later) for each $(u(t),z) \in X_n$ there exists a unique smooth path $z(t)\in \Dbarstar$.
We can then define the family of Lie groupoids over the punctured disk
\begin{equation*}
\begin{array}{cc}
\mathcal{G}_n' & \coloneq X_n/\!\sim\\
\downdownarrows & \\
\Dbarstar &
\end{array}
\end{equation*}
quotienting by homotopies $u(t,s)$ that fix the endpoints and such that $u(t,s)\cdot z^n \neq 0$, $\forall\; t,s$.
The source and target maps are respectively
\begin{equation*}
\sigma'([u,z])=z,
\qquad
\tau'([u,z])=z(1),
\end{equation*}
and multiplication is induced by concatenation of paths in $\psu$. Our first main result establishes an identification of this Lie groupoid with the \v{S}evera--Weinstein groupoid:

\begin{maintheorem}
The Lie groupoid
\begin{equation*}
\mathcal{G}_n'\rightrightarrows \Dbarstar
\end{equation*}
is the source-simply  connected Lie groupoid integrating the action Lie algebroid
\begin{equation*}
\mathcal{A}_n=\mathfrak{sl}_2(\mathbb R)\ltimes \Dbarstar.
\end{equation*}
\end{maintheorem}
Let $\mathcal{H}\in \mathfrak{sl}_2(\mathbb R)$ be the standard elliptic generator of $\mathfrak{sl}_2(\mathbb{R})$. We define
\begin{equation*}
\mathcal{G}_n:=X_n/\sim_n,
\end{equation*}
where the equivalence relation $\sim_n$ in addition to $\sim$ includes a composition of the path $u(t)$ with the path $v(t)=e^{\pi n\mathcal{H}t}$, that is the  loop winding $n$-times around the origin. 

\begin{maintheorem}
The quotient
\begin{equation*}
\mathcal{G}_n\rightrightarrows \Dbarstar
\end{equation*}
is a Lie groupoid integrating
\begin{equation*}
A_n=\mathfrak{sl}_2(\mathbb R)\ltimes \Dbarstar.
\end{equation*}
Moreover, $\mathcal{G}_n\rightrightarrows \Dbarstar$ is not an action groupoid over $\Dbarstar$, but its restriction to the boundary is:
\begin{equation*}
\mathcal{G}_n|_{\partial \Dbarstar}
\cong
\psun\ltimes S^1 .
\end{equation*}
\end{maintheorem}

The paper is organized as follows. In Section \ref{section2} we introduce hyperbolic metrics with (conical) singularities, and derive the infinitesimal isometries of such metrics. The resulting singular \(\mathfrak{sl}_2(\mathbb R)\) vector fields on the disk enable us to define the action Lie algebroid \(\mathcal{A}_n\), setting the ground for integration questions. In Section \ref{section3} we present a short review on Lie groupoids, Lie algebroids, and how the standard integration is carried through. In Section \ref{section4} we construct a more explicit integrating groupoid, \(\mathcal{G}_n\), prove its relation with the \v{S}evera--Weinstein groupoid, analyze its source fibers, and show that it is not an action groupoid on the punctured disk. Finally, in  Section \ref{section5} we explore an alternative viewpoint using correspondences. We embed the ramified Riemann surface defined into the product $\Dbar \times \Dbar$ as a correspondence. One can then establish a forgetful map from the groupoid $\mathcal{G}_n$ to the correspondence that preserves the product structure.

\paragraph{Acknowledgments} 
I would like to thank my supervisor Anton Alekseev for his valuable ideas and support throughout this project. I am grateful to Pierre de la Harpe, Eckhard Meinrenken, Muze Ren, Pavol \v{S}evera, Samson Shatashvili and Donald Youmans for the inspiring discussions and useful comments. I would moreover like to thank Nikolai Perry and Davide Saccardo for proofreading this manuscript and for their insightful inputs.
This research was supported by the grant 10005178 of the Swiss National Science Foundation (SNSF).
\section{Setup}\label{section2}
Consider the Poincaré disk model with boundary
\begin{equation*}
\Dbar=\{z\in \mathbb{C}\mid 0 \leq|z|\leq1\},
\end{equation*}
 equipped with the standard hyperbolic metric:
\begin{equation}\label{eq:PoincareDiskMetric}
ds^2=\frac{4\,|dz|^2}{(1-|z|^2)^2}, \quad z\in \Dbar.
\end{equation}
The boundary component is regarded as a boundary at infinity, and it is called an \textit{ideal boundary}. The isometries of such a metric are given by  $\psu$ acting by Möbius transformations 
\begin{equation}\label{mobius}
    g \cdot z \coloneqq \frac{a z + b}{\bar{b} z + \bar{a}}
\end{equation}
where $g = \begin{pmatrix} a & b \\ \bar{b} & \bar{a} \end{pmatrix}$  $\in \psu$, i.e. $\abs{a}^2-\abs{b}^2=1$, and $z \in \Dbar$.

Let us define the punctured Poincaré disk with boundary. 
Topologically, it is homeomorphic to the closed unit disk with a point $p$ removed
\begin{equation*}
    \Dbarstar \coloneqq \Dbar \setminus \{ p\}.
\end{equation*}
Geometrically, the puncture is encoded in a singularity of the metric\footnote{Some authors, for instance McOwen \cite{McOwen}, make a distinction between singularities and degeneracies: singularities are obtained by pushforward of the standard metric, while degeneracies result from pullback. Here, however, we will not maintain this distinction and will use the same term for both.}, in the following sense:

\begin{definition}
    A (conformal) metric $g$ on a surface $\Sigma$ admits a \emph{singularity of order $\beta-1, \beta \in \R $}, at $p \in \Sigma$ if it can be locally written as
\begin{equation*}
g = e^{2u} |z|^{2(\beta-1)} |dz|^2,
\end{equation*}
with 
\begin{equation*}
    u \in L^1, \text{ and }\text{ }  \Delta u \in L^1,
\end{equation*}
where $z$ is a local (conformal) coordinate such that
$z(p)=0$. 
\end{definition}
In particular, if $\beta \geq 1$, the angle $2\pi(1-\beta)\leq 0$ and this describes a cone. Then  $p$ is called a \textit{conical singularity} of $g$. This will be the case of interest throughout the paper. The conical singularities arise from local covering maps \cite{McOwen}:
\begin{equation}\label{nfold}
      p_\beta : z \longmapsto w = z^{\beta}.
\end{equation}
Fixing $\beta=n$, $n\in \mathbb{Z}_{>0}$,  
the pullback of the standard metric  \eqref{eq:PoincareDiskMetric} reads
\begin{equation*}
   g_{\Dbarstar}
   = p_{n}^{*} g_{\Dbar}
   = \frac{4n^{2}\,|z|^{2(n-1)}}{\bigl(1-|z|^{2n}\bigr)^{2}}\,|dz|^{2}.
\end{equation*}
Observe that this metric has a singularity of order $n-1$ at $z=0$. In this chosen local coordinate, the model space is then $\Dbarstar=\{z\in \mathbb{C}\mid 0 <|z|\leq1\}$.

The goal of this work is to study the isometries of such a singular hyperbolic metric. Since we are dealing with a ramified cover, we do not expect the $\psu$ action on the disk $\Dbar$ to lift to a $\psun$ action on the $n$-fold cover of $\Dbar$ by $p_n$, as mentioned in the introduction.
Before describing this failure explicitly, let us first explain how to define a $\psun$ action on the boundary circle. Consider the diagram:
 \begin{equation*}
\begin{tikzcd}
S^1 \arrow[r, "g^n"] \arrow[d, "p_n"'] & S^1 \arrow[d, "p_n"] \\
S^1 \arrow[r, "g"'] & S^1
\end{tikzcd}
\end{equation*}
where ${g}^n$ is the lift of $g$ such that $p_n \circ g^n=g \circ p_n$. The lifting can be realized explicitly in the following way. $\psu$ acts on the unit disk as in \eqref{mobius}, and in particular on boundary points $z=e^{i \theta}\in S^1$
\begin{equation*}
g\cdot e^{i\theta}
=
\frac{ae^{i\theta}+b}{\overline b e^{i\theta}+\overline a}.
\end{equation*}
This action preserves $S^1$ and is orientation-preserving.
Furthermore, it induces a circle diffeomorphism:
\begin{align*}
\begin{split}
    f_g:S^1 &\to S^1 \\
    e^{i \theta}& \mapsto \frac{ae^{i\theta}+b}{\overline b e^{i\theta}+\overline a} 
    \end{split}
\end{align*}
 which admits a lift to the real line
\begin{equation*}
F_g:\mathbb R\to \mathbb R
\end{equation*}
such that

\begin{equation}
\label{nlift}
e^{iF_g(\theta)}=f_g(e^{i\theta}).
\end{equation}
This lift $F_g$ is the unique up to integer multiples of $2\pi$:
\begin{equation*}
F_g\sim F_g+2\pi k,
\qquad k\in\mathbb Z.
\end{equation*}
Then an element $g^n\in$ $\psun$ over $g\in\psu$ is represented by a pair $g^n=(g,F_g)$
acting on the circle
\begin{equation*}
{
g^n\cdot e^{i\phi}
=
\exp\left(\frac{i}{n}F_g(n\phi)\right).
}
\end{equation*}
Using this and \eqref{nlift} we can then write  \cite{Dai2000}
\begin{equation}
\label{nformula}
    g^n\cdot z
=
\left(
\frac{a z^n+b}{\overline b z^n+\overline a}
\right)^{1/n}, \quad |z|=1 
\end{equation}
where the choice of the $n$-th root is determined by the choice of the lift $F_g$. 
By continuity, once such a lift is fixed, the action is globally well-defined and compatible with the cyclic order on the circle.
This formula fails to extend to the interior of the disk. In contrast to the circle action, the $n$-th root cannot be defined as a single-valued function on $\Dbar$ due to the lack of the cyclic ordering on the disk. Nevertheless, the formula $\eqref{nformula}$ is locally well-defined.
Given the branched covering map $p_n: z \mapsto z^n$, since $p_n'(z) = nz^{n-1} \neq 0$ in $\Dbarstar$, the map $p_n$ is a local diffeomorphism. Thus every point in $\Dbarstar$ has a
sufficiently small  neighborhood $U \subset \Dbarstar$ on which a smooth branch of $p_n^{-1}$ can be chosen.
 This enables us to formulate the next proposition.
\begin{proposition}\label{vfn}
    The following vector fields in $\Dbarstar$
 \begin{align*}
    E &= \frac{1}{n}((z^{1-n} - z^{1+n})\partial_{z} + (\overline{z}^{1-n} - \overline{z}^{1+n})\partial_{\overline{z}}) \\
    F &= \frac{i}{n}((z^{1-n} + z^{1+n})\partial_{z} -(\overline{z}^{1-n} + \overline{z}^{1+n})\partial_{\overline{z}}) \\
    H &= \frac{2i}{n}(z\partial_{z} - \overline{z}\partial_{\overline{z}})
\end{align*}
  span a Lie algebra isomorphic to $\mathfrak{sl}_2(\mathbb{R})$ and provide the infinitesimal action for the formula \eqref{nformula}. 
\end{proposition}
\begin{proof}
(Explicit calculation.)
Consider a smooth curve in the group:
\begin{align*}
    \Gamma_{(a,b)}&:[0,1]\xrightarrow{} \psu\\
    t& \mapsto \begin{pmatrix}
        a_t & b_t \\
        \bar{b}_t & \bar{a}_t
    \end{pmatrix},
\end{align*}
such that $\Gamma_{(a,b)}(0)=e \in \psu$, i.e. $a_0=1$ and $b_0=0$. Expansion near identity gives:
\begin{equation}
a_t=1+\dot{a_0} t+O(t^2),
\qquad
b_t=\dot{b_0} t+O(t^2).
\end{equation}
 Then
\begin{equation*}
\bar a_t=1+\bar{\dot{a_0}} t+O(t^2),
\qquad
\bar b_t=\bar{\dot{b_0}} t+O(t^2),
\end{equation*}
and the group condition $|a|^2-|b|^2=1$ yields $\dot{a_0}=i\alpha$ and $\dot{b_0}=\beta$, with $ \alpha \in \mathbb{R}$ and $\beta \in \mathbb{C}$.

Now let us go back to the formula \eqref{nformula}. Define
\begin{equation}
\gamma_{(a_t,b_t)}(z)
\coloneqq
\left(\frac{a_t z^n+b_t}{\bar b_t\, z^n+\bar a_t}\right)^{1/n}.
\end{equation}
As we explained earlier, while globally this formula remains ill-defined, it locally exists (for $\frac{a_tz^n+b_t}{\bar{b_t}z^n+\bar{a_t}} \neq0$) and it defines $n$ different functions in the neighbourhood of $\frac{a_tz^n+b_t}{\bar{b_t}z^n+\bar{a_t}}$. We pick the branch corresponding to the $n$-th root close to identity, so that $\gamma_{(1,0)}(z)=z$. We can then differentiate to obtain vector fields, as follows.
Set
\begin{equation}\label{Rt}
R_t(z):=\frac{a_t z^n+b_t}{\bar b_t\, z^n+\bar a_t}.
\end{equation}
Since $R_0(z)=z^n$,  differentiating (\ref{Rt}) and evaluating at $t=0$ gives
\begin{equation}
    \begin{aligned}
\frac{d}{dt}\Big|_{t=0} R_t(z)
&=(\dot a_0 z^n+\dot b_0)-z^n(\dot{\bar b}_0 z^n+\dot{\bar a}_0) \\
&= i\alpha z^n+\beta-z^n(\bar\beta z^n-i\alpha) \\
&= \beta+2i\alpha z^n-\bar\beta z^{2n}.
\end{aligned}
\end{equation}
Now use the chain rule for the $n$-th root to obtain:
\begin{equation*}
    \begin{aligned}
   \frac{d}{dt}\Big|_{t=0} \gamma_{a_tb_t}(z)
&=\frac1n z^{1-n}\bigl(\beta+2i\alpha z^n-\bar\beta z^{2n}\bigr) \\
&=\frac1n\bigl(\beta z^{1-n}+2i\alpha z-\bar\beta z^{n+1}\bigr).
\end{aligned}
\end{equation*}
Hence the corresponding real infinitesimal vector field is
\begin{equation*}
V_{\alpha,\beta}
=
\frac1n\bigl(\beta z^{1-n}+2i\alpha z-\bar\beta z^{n+1}\bigr)\partial_z
+
\frac1n\bigl(\bar\beta \bar z^{1-n}-2i\alpha \bar z-\beta \bar z^{n+1}\bigr)\partial_{\bar z}.
\end{equation*}
Choosing $(\alpha,\beta)=(0,1),(0,i),(1,0)$, one obtains
\begin{equation}
\begin{aligned}\label{vf}
E
&=
\frac1n\Bigl((z^{1-n}-z^{n+1})\partial_z+(\bar z^{1-n}-\bar z^{n+1})\partial_{\bar z}\Bigr), \\
F
&=
\frac1n\Bigl(i(z^{1-n}+z^{n+1})\partial_z-i(\bar z^{1-n}+\bar z^{n+1})\partial_{\bar z}\Bigr), \\
H
&=
\frac{2i}{n}\Bigl(z\partial_z-\bar z\partial_{\bar z}\Bigr).
\end{aligned}
\end{equation}
The Lie bracket reads
\begin{equation}
[E,F]=2H,
\qquad
[H,E]=-2F,
\qquad
[H,F]=2E.
\end{equation}
Thus these vector fields form a copy of $\mathfrak{su}(1,1)\cong \mathfrak{sl}_2(\mathbb R)$.
Observe that for $n>1$, the term $z^{1-n}$ is singular at $z=0$, so these are naturally vector fields on the punctured disk. 
\end{proof}
\begin{remark}
    Note that using polar coordinates $z=\rho e^{i\phi}$ and restricting to the boundary $\rho=1$, one recovers the familiar $n$-fold Möbius infinitesimal action on the circle \cite{Valach_2020}:
\begin{equation}
\begin{aligned}
    E&=-\frac{2}{n}\sin(n\phi)\partial_\phi \\
    F&=\frac{2}{n}\cos(n\phi) \partial_\phi\\
    H&=\frac{2}{n}\,\partial_\phi.
\end{aligned}
\end{equation}
\end{remark}
Because of their singular behaviour at $z=0$, these vector fields do not integrate into a Lie group action.
We can instead work in a Lie algebroid framework.
For all $n \in \Z_{>0}$, we can define Lie algebroids:
\begin{equation} \label{actionLiealg}
    \mathcal{A}_{n} := \slR \ltimes \Dbarstar
\end{equation}
with anchor map 
\begin{equation*}
\rho_{n}: \mathcal{A}_{n} \to T\Dbarstar,
\end{equation*}
given by the vector fields in \eqref{vf}.
Note that $\rho_1$ corresponds to Möbius transformations and it extends to the whole disk.

Our goal (and the structure of this manuscript) can then be resumed with the following questions:
\begin{question}\label{question1}
    Is the Lie algebroid $\mathcal{A}_n$ integrable?
\end{question}
\begin{question}\label{question2}
    If so, what is the integrating Lie groupoid?
\end{question}

\begin{question}
    How to recover $\psun$ action?
\end{question}
   
\section{Preliminaries on Lie groupoids and algebroids}\label{section3}
We will review in this paragraph some background on Lie algebroids and groupoids, illustrated by some warm-up examples. Most of it is based on the lecture notes \cite{meinrenkengroupoids}, and on the resume of \cite{bursztyn2023liegroupoids}.

\subsection{Lie groupoids}
 Lie groupoids are the many-object generalizations of a Lie group. They are described by a manifold $M$ of objects, together with a collection of arrows (possibly empty) assigned to any two objects $m_1, m_2 \in M$. We will require that these collections of arrows fit smoothly together into a manifold of arrows, with composition defined provided that the end point of one arrow is the starting point of the next. We will draw the arrows from right to left as this choice is more convenient when picturing composition:

\begin{equation*}
\begin{tikzcd}
\tau(g) & \sigma(g) \arrow[l, bend right=35, "g" above]
\end{tikzcd}
\end{equation*}

\begin{definition}
A \textit{Lie groupoid} $\mathcal{G}\rightrightarrows M$ \footnote{The two arrows in this notation represent the source and the target maps.} consists of a manifold of arrows $\mathcal{G}$ and a submanifold $\iota: M \hookrightarrow \mathcal{G}$ of units (also called objects) equipped with two surjective submersions $\sigma,\tau$ called source and target maps such that 
\begin{equation}
    \sigma \circ \iota = \tau \circ \iota = \text{Id}_M,
\end{equation}
and a multiplication map $m:(g_1,g_2)\mapsto g_1g_2$ defined whenever $g$ and $h$ are \textit{composable}, i.e. $\sigma(g_1)=\tau(g_2)$. 
The composition rule can be visualized as follows:
\begin{center}
\begin{minipage}{.3\textwidth}
\begin{tikzcd}
m_3 & m_2 \arrow[l, bend right=35, "g_1" above] & m_1 \arrow[l, bend right=35, "g_2" above]
\end{tikzcd}
\end{minipage}\hspace{0.5cm}
\begin{minipage}{.5\textwidth}
\begin{tikzcd}
m_3 & m_1 \arrow[l, bend right=35, "g_1 \circ g_2" above]
\end{tikzcd}
\end{minipage}
\end{center} 

This data must satisfy the axioms of associativity, existence of a unit (trivial arrow) and an inverse element (reverse arrow).
\end{definition}

\begin{example}{(Lie groups).} A Lie group $G$ is a groupoid with a single object $\mathcal{G}\rightrightarrows \text{pt}$.
\end{example}

 \begin{example}
     (Manifolds). On the other hand, every manifold $M$ can be seen as the trivial Lie groupoid $M \rightrightarrows M$, where all elements are also units. The multiplication law is then: $m=m_1 \circ m_2 \iff m=m_1=m_2$.
 \end{example}

 \begin{example}{(Pair groupoid).} To any manifold $M$ one can associate the pair groupoid $\text{Pair}(M)=(M \times M)\rightrightarrows M$, with a unique arrow between any two points $m, m'$. Source and target are $\sigma((m,m'))=m$ and $\tau((m,m'))=m'$. The multiplication law is:
 \begin{equation*}
    (m, m')=(m_1, m_1')\circ (m_2, m_2')\iff m=m_1,\quad m'=m_2', \quad m_1'=m_2. 
 \end{equation*} 
     
 \end{example}

\subsection{Lie algebroids}
The infinitesimal data of a Lie groupoid is captured by Lie algebroids.
\begin{definition}
    A \textit{Lie algebroid} is a vector bundle $\pi:\mathcal{A} \rightarrow M$, together with a Lie bracket $[\cdot,\cdot]$ on its space of sections, such that there exists a vector bundle map
    \begin{equation*}
        \rho: \mathcal{A} \rightarrow TM
    \end{equation*}
    called the \textit{anchor map}, satisfying the Leibniz rule
    \begin{equation*}
        [\alpha, f\beta]=f[\alpha,\beta]+ (\rho(\alpha)f)\beta,
    \end{equation*}
    $\forall \alpha, \beta \in \Gamma(\mathcal{A})$ and $f \in C^\infty(M).$
 \end{definition}
\begin{example}{(Lie algebra).}
    A Lie algebroid over a point $\mathcal{A}\rightarrow \text{pt}$ is the same as a Lie algebra.
\end{example}
\begin{example}{(Action Lie algebroid).}
   A Lie algebra $\mathfrak{g}$ action on $M$ defines a Lie algebroid $\mathcal{A}=\mathfrak{g}\ltimes M$. The anchor is given by the action map
   \begin{equation*}
   \begin{split}
        M \times \mathfrak{g} &\xrightarrow[]{} TM \\
        (m, X)&\mapsto X_M(m).
   \end{split}
    \end{equation*}
The Lie bracket is the unique extension of the given Lie bracket on constant sections determined by the Leibniz rule:
\begin{equation*}
    [X,Y]=[X,Y]_\mathfrak{g}+\mathcal{L}_{\rho(X)}Y - \mathcal{L}_{\rho(Y)}X,
\end{equation*}
for all $X,Y: M \xrightarrow[]{}\mathfrak{g}$  sections of $\mathcal{A}$, where $\mathcal{L}_{\rho(X)}$ denotes the Lie derivative along the vector field defined by the anchor. The Lie algebroid we defined in \eqref{actionLiealg} is an example of an action Lie algebroid.
\end{example}
\subsection{Integrability theory}

Let $\mathcal{G} \rightrightarrows M$ be a Lie groupoid, with source and target maps denoted $\sigma,\tau \colon \mathcal{G} \to M$.
A vector field $X \in \mathfrak{X}(\mathcal{G})$ is said to be \emph{left-invariant} if it is tangent to the $\tau$-fibers, and
\begin{equation*}
X_g = (L_g)_* X_{\sigma(g)}
\end{equation*}
for all $g \in G$, where
\begin{equation*}
L_g \colon \tau^{-1}(\sigma(g)) \to \tau^{-1}(\tau(g))
\end{equation*}
is the left translation. 
Given these data, we can then construct the Lie algebroid
\begin{equation*}
\mathcal{A} = \text{Lie}(\mathcal{G}) \rightarrow M
\end{equation*}
of a Lie groupoid. It is defined by the following data:
\begin{itemize}
    \item A vector bundle, $\mathcal{A} = \ker(T\tau)\big|_M$.
    \item The bracket on $\Gamma(\mathcal{A})$ is identified with the left-invariant vector fields.
    \item The anchor is the restriction of $T\sigma \colon TG \to TM$ to $\mathcal{A} \subseteq TG|_M.$
\end{itemize}
A Lie algebroid $\mathcal{A} \rightrightarrows M$ is called \emph{integrable} if it is of the form $\mathcal{A} = \text{Lie}(\mathcal{G})$ for a Lie groupoid $\mathcal{G} \rightrightarrows M$.
Similar to the classical Lie theory reviewed in the introduction, the integrability of Lie algebroids is ruled by analogues of Lie's theorems.
Lie's first and second theorems admit a promising generalization to groupoids and algebroids, namely \cite{crainic2004integrabilityliebrackets}:

\begin{theorem}[Higher Lie I]
    If $\mathcal{A}$ is an integrable Lie algebroid, then there exists a unique (up to isomorphism) source-simply connected Lie groupoid integrating it.
\end{theorem}

\begin{theorem}[Higher Lie II]
    Let $\phi: \mathcal{A} \to \mathcal{B}$ be a morphism of Lie algebroids, and let $\mathcal{G}$ and $\mathcal{H}$ be two Lie groupoids integrating $\mathcal{A}$ and $\mathcal{B}$ respectively. If $\mathcal{G}$ is source-simply connected, then there exists a (unique) morphism of Lie groupoids $\Phi: \mathcal{G} \to \mathcal{H}$ integrating $\phi$.
\end{theorem}
On the other hand, Lie's third theorem does not always hold true for the higher case. While it was earlier shown in \cite{AM85} that not every Lie algebroid admits an integration into a Lie groupoid, the (computable) obstructions were spelled out in \cite{crainic2004integrabilityliebrackets}:

\begin{theorem}[Higher Lie III]
    Let $\mathcal{A} \to M$ be a Lie algebroid over $M$.
    $\mathcal{A}$ is integrable if and only if  the monodromy groups $N_{x}(\mathcal{A})$ are locally uniformly discrete, $\forall x \in M$.
\end{theorem}
We refer the reader to \cite{crainic2004integrabilityliebrackets} for the technical details and the proof of this theorem. Meanwhile, 
in the special case of action Lie algebroids, this theorem simplifies to the following result, proved earlier by \cite{DAZORD199777}:
\begin{theorem}\label{actionlie}
    Every action Lie algebroid is integrable.
\end{theorem}
This theorem allows to affirmatively answer Question \ref{question1}. 
In the next paragraph, we briefly review the classical construction of the integrating groupoid \textit{à la} \v{S}evera--Weinstein \cite{Severa01}. 
\begin{definition}\label{apathhomotopy}
    Let $\pi:\mathcal{A}\rightarrow M$ be a Lie algebroid with anchor map $\rho$.
    An $\mathcal{A}$-path is a $C^1$ curve $a: I \to \mathcal{A}$ such that:
    \begin{equation*}
    \rho(a(t)) = \frac{d}{dt} \pi(a(t)).
    \end{equation*}
\end{definition}
We will denote the set of $\mathcal{A}$-paths by $\mathcal{P}(\mathcal{A})$. It is useful to think of an $\mathcal{A}$-path as the Lie algebroid morphism
\begin{equation}
    a dt:TI \rightarrow \mathcal{A}.
\end{equation}
One can then define an \textit{$\mathcal{A}$-path homotopy} between two $\mathcal{A}$-paths $a_1:TI \rightarrow \mathcal{A}$ and $a_2:TI \rightarrow \mathcal{A}$ to be the Lie algebroid morphism \cite{fernandes2021localglobalintegrabilitylie}
\begin{equation*}
    h:T(I\times I)\xrightarrow{}\mathcal{A},
\end{equation*}
satisfying the boundary conditions:
\begin{equation}\label{bc}
h|_{TI\times \{0\}}=a_1, \quad h|_{TI\times \{1\}}=a_2, \quad h|_{\{0\}\times TI}=h|_{\{1\}\times TI}=0. 
\end{equation}

\begin{theorem}[{{\cite{crainic2004integrabilityliebrackets}}}]
\label{severa-weinstein}
    Whenever a Lie algebroid $\mathcal{A}$ is integrable, the quotient 
\begin{equation}
    \mathcal{G}(\mathcal{A}) = \frac{\mathcal{P}({\mathcal{A})}}{\mathcal{A}{\text{-path homotopy }}}
\end{equation}
admits a smooth structure, and it is the unique (up to isomorphisms) source-simply connected Lie groupoid integrating $\mathcal{A}$.
\end{theorem}

 This theorem provides us with $\mathcal{G}(\mathcal{A}_n)$, the unique source-simply connected Lie groupoid integrating the action Lie algebroid $\mathcal{A}_n$.
 \begin{remark}
     Note that the groupoid $\mathcal{G}(\mathcal{A})$ can always be constructed without any integrability criteria, and it produces a topological groupoid only. Obstructions may show up when one tries to equip it with a smooth structure, and that is when the integrability criteria shows up. This construction was suggested by Cattaneo--Felder \cite{Cattaneo_2001} for cotangent Lie algebroids, and then by \v{S}evera and Weinstein \cite{Severa01} for general Lie algebroids.
 \end{remark}
 
\section{An explicit construction of the integrating groupoid}\label{section4}
While the \v{S}evera--Weinstein construction guarantees the existence of a source-simply connected groupoid integrating $\mathcal{A}_n$, it remains in the present case too abstract to make the $n$-fold Möbius symmetry manifest. We therefore propose a more explicit model for the integrating groupoid, built directly from paths in $\psu$. Beyond providing a concrete description, this model clarifies precisely how the obstruction to extending the $\psun$ action from the boundary to the interior of the punctured disk is encoded in the groupoid structure.

\subsection{Integrating groupoid}
\begin{definition}
    Define $X_n$ to be the space of pairs consisting of a smooth path in $\psu$ and a point on the punctured disk, subject to conditions:
    \begin{equation*}
    X_{n} := \{ (u: [0,1] \to \psu, z \in \Dbarstar) \mid u(0)=e, u(t) \cdot z^{n} \neq 0 \}.
    \end{equation*}
\end{definition}
\begin{proposition}\label{lift}
    For every $(u(t), z) \in X_n$, there exists a unique smooth path $z(t) \in \Dbarstar$ such that:
    \begin{equation*}
    z(0) = z \quad \text{and} \quad z(t)^n = u(t) \cdot z^n
    \end{equation*}
\end{proposition}

\begin{proof}[Proof]
    Let $p: \Dbarstar \to \Dbarstar$ be the $n$-fold covering map: $p(z)=z^n$, with initial value $z(0)=z$.
    By the Path Lifting Lemma (see diagram below), there exists a lift $z(t)$ such that $p(z(t)) = u(t) \cdot z^n=z(t)^n$.
    Uniqueness is guaranteed since we fix $z(0)$.

\begin{equation*}
\begin{tikzcd}
& {\Dbarstar} \arrow[d, "p"] \\
I=[0,1] \arrow[ur, "{z(t)}"] \arrow[r, "\gamma"'] & \Dbarstar
\end{tikzcd}
\end{equation*}

\end{proof}
\begin{definition}
\label{homotopy}
   We consider a family of Lie groupoids
   \begin{equation}
       \mathcal{G}_{n}' = X_{n} / \sim'
   \end{equation} where the equivalence relation 
    $ \sim'$ is defined as follows:
\begin{equation}
(u(t),z)\sim'(v(t),w)
    \end{equation}
    if and only if
$w=z$ and there exists a smooth homotopy $u(t,s)$ fixing the endpoints such that: 
\begin{equation*}
   u(t,0) = u(t), \quad  u(t,1) = v(t), \quad u(t,s)\cdot z^n \neq 0, \forall\ t,s.
\end{equation*}
The source and target maps of $\mathcal{G}_{n}'$ are defined by the formula
    \begin{equation*}
        \sigma'(u(t), z) = z, \quad \tau'(u(t), z) = z(1).
    \end{equation*}
\end{definition}
Given two equivalence classes of paths in $\mathcal{G}_n'$, $[(u_1,z)]$ and $ [(u_2,w)]$, such that they are composable, i.e. $\tau'(u_1(t),z)=w=z(1)$, we can define a product rule in $\mathcal{G}_{n}'$, given by concatenation of (homotopy classes of) paths:
 \begin{equation*}
 [(u_2,w)] \circ [(u_1,z)]\ = [(u_2\circ u_1,z)],
 \end{equation*}
 with\footnote{For notation simplicity, we will remove the square brackets from now on, but we will always consider composition of paths under homotopy comprised.}
 \begin{equation}\label{pathconcat}
 (u_2\circ u_1)(t)=
   \begin{cases}
      u_1(2t) & t \in [0,\frac{1}{2}]\\
     u_2(2t-1). u_1(1) & t \in [\frac{1}{2}, 1]
    \end{cases},
\end{equation}
where $.$ is the matrix multiplication in $\psu$.
This composition is associative up to reparametrization (which is a homotopy)\footnote{Note that concatenation of smooth paths does not have to be smooth; one can remedy this by reparametrizing paths so that they are constant near endpoints, see \textit{sitting endpoints, \cite{meinrenkengroupoids}}.}. Indeed, 
\begin{equation}
    (u_3 \circ u_2) \circ u_1=  \begin{cases}
      u_1(4t) & t \in [0,\frac{1}{4}]\\
      u_2(4t-1).u_1(1) & t \in [\frac{1}{4}, \frac{1}{2}]\\
      u_3(2t-1).u_2(1).u_1(1)  & t \in [\frac{1}{2}, 1]
      
    \end{cases}
\end{equation}

and 
\begin{equation}
     u_3 \circ (u_2 \circ u_1)=  \begin{cases}
      u_1(2t) & t \in [0,\frac{1}{2}]\\
     u_2(4t-2). u_1(1) & t \in [\frac{1}{2}, \frac{3}{4}]\\
    u_3(4t-3).u_2(1).u_1(1)& t \in [\frac{3}{4}, 1]\\
      
    \end{cases}
\end{equation}
It is easy to see that these two compositions are the same upon the time reparametrization:
\begin{equation*}
\varphi:[0,1]\to[0,1],\qquad
\varphi(t)=
\begin{cases}
2t, & t\in\left[0,\frac14\right],\\
t+\frac14, & t\in\left[\frac14,\frac12\right],\\
\frac{t+1}{2}, & t\in\left[\frac12,1\right].
\end{cases}
\end{equation*}
\begin{theorem}\label{gn'}
    The space $\mathcal{G}_n'=X_n/\sim'$ with source and target maps $\sigma'$ and $\tau'$, and multiplication map $\circ$, is the unique up to isomorphism source-simply connected Lie groupoid integrating the Lie algebroid $\mathcal{A}_n.$
\end{theorem}
In order to prove this theorem, we first establish the following intermediate proposition.
Let 
\begin{equation*}
    \mathcal{G}(\mathcal{A}_n)=\mathcal{P}(\mathcal{A}_n)/\sim_{\mathcal{A}_n}
\end{equation*}
be the \v{S}evera--Weinstein groupoid corresponding to the Lie algebroid $\pi:\An\rightarrow \Dbarstar$. The equivalence $\sim_{\mathcal{A}_n}$ here stands for $\A_n$-path homotopy as in Definition \ref{apathhomotopy}. Since $\An=\mathfrak{sl}_2(\mathbb{R}) \ltimes \Dbarstar$, every $\An$-path can be written as $a(t)=(X(t),z(t))$, where $X(t)\in \mathfrak{sl}_2(\mathbb{R})$ and $z(t)\coloneq \pi(a(t)) \in \Dbarstar$ is the base path.
\begin{proposition}\label{isoSW}
The map $\Phi:\mathcal{G}_n'\longrightarrow\mathcal{G}(\mathcal{A}_n),$ defined as
\begin{equation*}
\Phi\big([(u,z)]\big)=\big[a(t)\big],
\end{equation*}
with $a(t)=(X(t),z(t))$ such that
\begin{align*}
z(t)^n=u(t)\cdot z^n,\qquad
X(t) &\coloneqq u'(t)u(t)^{-1},
\end{align*}
is a Lie groupoid isomorphism.
\end{proposition}

\begin{proof}
We proceed in three steps. Firstly we prove that there is a bijection between $X_n$ and \(\mathcal{P}(\mathcal{A}_n)\), prior to quotienting by homotopies.
For \((u,z)\in X_n\), Proposition \ref{lift} gives the lift $z(t)$ with \(z(0)=z\) and $z(t)^n=u(t)\cdot z^n$. Differentiating $z(t)^n=u(t)\cdot z^n$ and using the same calculation as in the proof of Proposition \ref{vfn} with $X(t)=u'(t)u(t)^{-1}$ the infinitesimal Möbius action, we obtain
\begin{equation*}
   \dot z(t)=\rho_n(u'(t)u(t)^{-1},z(t)). 
\end{equation*}
Thus the curve $a(t)$ satisfies the anchor condition hence it is an $\mathcal{A}_n$-path. Conversely, let $a(t)=(X(t),z(t))$ be an $\mathcal{A}_n$-path. Let $u(t)$ be the unique solution of the ODE 
\begin{equation*}
u'(t)u(t)^{-1}=X(t), \qquad u(0)=e.
\end{equation*}
The anchor condition implies that $z(t)^n$ and $u(t)\cdot z(0)^n$ solve the same ODE with the same initial conditions. By uniqueness of solutions, they are then equal: $z(t)^n= u(t)\cdot z(0)^n$.
Because $\A_n$ is a Lie algebroid over $\Dbarstar$, its base path $z(t)\neq 0$, $\forall t$. This means that $z(t)^n\neq 0$ and therefore $u(t)\cdot z(0)^n\neq0$, $\forall t$. Hence $(u(t),z(0))\in X_n$. This gives a bijection $\bar{\Phi}:X_n\to\mathcal{P}(\mathcal{A}_n)$.

Secondly, we prove that the bijection descends to homotopy classes. Suppose that $(u,z) \sim' (v,z)$ and let $u(t,s)$ be the corresponding homotopy fixing the endpoints. Let then $z(t,s)$ be the lift determined by
\begin{equation*}
    z(0,s)=z, \qquad u(t,s)\cdot z^n=z(t,s)^n.
\end{equation*}
Define now a map $h: T(I \times I)\rightarrow \mathcal{A}_n$ as follows:
\begin{equation}
    h(\partial t)=(\partial_t u\, u^{-1}, z(t,s)) \qquad h(\partial_s)=(\partial_s u\, u^{-1}, z(t,s)).
\end{equation}
Differentiating $z(t,s)^n=u(t,s)\cdot z^n$
in both variables gives the anchor map equation for $h$. Commutativity of mixed derivatives allows for $h$ to satisfy the compatibility condition in \cite{crainic2004integrabilityliebrackets}, and hence it is a Lie algebroid morphism. Since $u(t,s)$ fixes the endpoints, its $s$-derivative vanishes at $t=0,1$, and so $h$ satisfies the boundary conditions in equation \eqref{bc}. It is therefore an $\mathcal{A}_n$ homotopy.
Conversely, given an $\mathcal{A}_n$-homotopy $h(t,s)$ with base path $z(t,s)$, integration for fixed $s$ yields a unique smooth path $u(t,s)$ with $u(0,s)=e$. The boundary conditions for $h(t,s)$ give:
\begin{equation*}
\partial_su(0,s)=\partial_su(1,s)=0,
\end{equation*}
so $u(t,s)$ has fixed endpoints. Similarly, using also the anchor equation, we get
\begin{equation*}
    \partial_sz(0,s)=0,
\end{equation*}
and hence we can fix $z(0,s)$ to be $z$. As before, for each $s$, $z(t,s)^n$ and $u(t,s)\cdot z^n$ solve the same ODE, and uniqueness gives $z(t,s)^n=u(t,s)\cdot z^n$.
Since $z(t,s)$ is a path in $\Dbarstar$, then the left hand side never vanishes. Therefore $u(t,s)$ is a homotopy as defined in Definition \ref{homotopy}.
Then $\bar\Phi$ descends to the bijection
\begin{equation*}
    \Phi:\mathcal{G}_n'=X_n/\!\sim'\;\longrightarrow\;\mathcal{P}(\mathcal{A}_n)/\!\sim_{\mathcal{A}_n}=\mathcal{G}_{}(\mathcal{A}_n).
\end{equation*}

Finally, let us verify that the groupoid structure is preserved under $\Phi$.
\begin{itemize}
\item \textbf{Source and target:} For $[(u,z)]\in\mathcal{G}_n'$, $\sigma'([(u,z)])=z$ and $\tau'([(u,z)])=z(1)$ with $z(1)^n=u(1)\cdot z^n$. In $\mathcal{G}(\An)$ we have $\sigma([a])=\pi(a(0))=z$ and $\tau([a])=\pi(a(1))=z(1)$. Hence $\Phi$ preserves source and target.
\item \textbf{Multiplication:} Let \(\alpha=[(u,z)]\) and \(\beta=[(v,w)]\) with \(\tau(\alpha)=w\). In \(\mathcal{G}_n'\) their composition  $\beta \circ \alpha=([(v\circ u,z)]$ is
as in equation \eqref{pathconcat}.

The images of $\alpha$ and $\beta$ under $\Phi$ are represented by $\mathcal{A}_n$-paths:
\begin{equation}
    a_u=(u'(t)u^{-1}(t), z(t))\quad \text{and}\quad a_v=(v'(t)v^{-1}(t), w(t)).
\end{equation}
The concatenation of these two paths is
\begin{equation*}
    (a_v\circ a_u)=\begin{cases}
(2u'(2t)u(2t)^{-1}, z(2t))&0\le t\le \frac12,\\
(2v'(2t-1)v(2t-1)^{-1},w(2t-1))&\frac12\le t\le1.
\end{cases}
\end{equation*}

Direct computation shows that the $\mathcal{A}_n$-path associated with the composition $\beta \circ \alpha$ is
\begin{equation*}
    ((v\circ u)'(t)(v\circ u)^{-1}(t),z(t)\circ w(t))=\begin{cases}
(2u'(2t)u(2t)^{-1},z(2t))&0\le t\le \frac12,\\
(2v'(2t-1)v(2t-1)^{-1},w(2t-1))&\frac12\le t\le1.
\end{cases}
\end{equation*}
Therefore
$\Phi(\beta\circ\alpha)=\Phi(\beta)\circ\Phi(\alpha)$.
\item Smooth structure:
Finally, since $\mathcal{A}_n$ is integrable, Theorem \ref{severa-weinstein} gives $\mathcal{G}(\mathcal{A}_n)$ its smooth structure. We can transport this smooth structure to $\mathcal{G}_n'$ through the bijection $\Phi$.
\end{itemize}
Consequently, $\Phi$ is a Lie groupoid isomorphism.
\end{proof}

\begin{proof}[Proof of Theorem \ref{gn'}]
  By the uniqueness (up to isomorphism) argument in Theorem \ref{severa-weinstein}, we can conclude that $\mathcal{G}'_n$ is the source simply connected groupoid integrating the action Lie algebroid $\mathcal{A}_n$.
\end{proof}
\begin{definition}[]
    We define now the groupoid $\mathcal{G}_{n} \coloneq X_{n} / \sim_{n}$.
    The equivalence relation $\sim_n$ includes composition with loops in $\psu$  winding $n$-times around the origin:
    \begin{equation*}
    \sim_{n} \coloneq \{ \sim , \circ\,  e^{\pi n\mathcal{H}t} \}
    \end{equation*}
    where $\mathcal{H} = \begin{pmatrix} 0 & -1 \\ 1 & 0 \end{pmatrix} \in \slR$ is the elliptic generator. The source, target and multiplication map descend to this quotient. $\mathcal{G}_n$ inherits a well-defined groupoid structure with source and target maps:
    \begin{equation}
        \sigma([u,z])=z, \hspace{9mm} \tau([u,z])=z(1).
    \end{equation}
\end{definition}

\begin{theorem}
\label{thm_G_n_integrates_A_n}
    The groupoid $\mathcal{G}_n$ integrates the Lie algebroid $\mathcal{A}_n.$
\end{theorem}
To prove this theorem, we need first the following preliminary Lemma:
\begin{lemma}\label{loops}
   Let $c_n(t) = e^{\pi n\mathcal{H}t}$ be the loop in $\psu$, such that the path $c_n(t)\cdot z^n$ winds $n$ times around the origin. For any $z \in \mathbb{D}_*$, the classes $[(c_n^k, z)], k\in \mathbb{Z}$,  form a central subgroupoid of $\mathcal{G}_n'$.
\end{lemma}
\begin{proof}
Consider any composable path $[(u, z)]$ in the groupoid. Because $c_n(t)$ represents the generator of the fundamental group $\pi_1(\psu) \cong \mathbb{Z}$, concatenating with $c_n(t)$ commutes up to path homotopy. Explicitly, checking the two compositions yields:
 \begin{equation}
   \gamma_0(t)\coloneqq u(t) \circ c_n(t)\  = 
    \begin{cases}
      c_n(2t) & t \in [0,\frac{1}{2}]\\
     u(2t-1) & t \in [\frac{1}{2}, 1]
    \end{cases}, 
    \end{equation}
and 
     \begin{equation}
 \gamma_1(t) \coloneq  c_n(t) \circ u(t)\ = 
    \begin{cases}
      u(2t) & t \in [0,\frac{1}{2}]\\
      c_n(2t-1).u(1) & t \in [\frac{1}{2}, 1]
    \end{cases}.   
    \end{equation}
Because $c_n(1) = e^{\pi n \mathcal{H}} = e \in \psu$, these two resulting concatenated paths are smoothly homotopic relative to their endpoints, with homotopy: 
\begin{equation*}
    F(s,t)= 
    \begin{cases}
        u(2t) & 0 \leq t \leq \frac{s}{2} \\
        c_n (2t-s).u(s) & \frac{s}{2} \leq t \leq \frac{s+1}{2}\\
        u(2t-1) & \frac{s+1}{2}\leq t \leq 1.
    \end{cases}
\end{equation*}
One can easily verify that $F(0,t)=\gamma_0$, $F(1,t)=\gamma_1$. It remains to show that it is an admissible homotopy. On the first and third pieces, the nonvanishing condition follows from $u(t)\cdot z^n \neq 0$. On the middle piece 
\begin{equation*}
    c_n(2t-s).(u(s)\cdot z^n)\neq 0
\end{equation*}
as $c_n$ acts by rotations. 
 \end{proof}

With this lemma, the proof of Theorem \ref{thm_G_n_integrates_A_n} follows easily:
\begin{proof}
First, note that the action of the path $c_n(t)$ leaves the endpoints fixed, meaning the source and target of $[(c_n, z)]$ are both $z$, i.e. it belongs to the isotropy group of $\mathcal{G}_n'$ at $z$. By Lemma \ref{loops}, the equivalence relation $\sim_n$ corresponds to quotienting $\mathcal{G}_n'$ by the central, discrete subgroupoid generated by $[(c_n, z)]$. The quotient of a Lie groupoid by a discrete, normal subgroup of its isotropy bundle inherits a unique smooth structure such that the projection map $\pi: \mathcal{G}_n' \to \mathcal{G}_n$ is a local diffeomorphism. 
Let us denote by 
$\mathcal{A}(\mathcal{G}_n')$ and $\mathcal{A}(\mathcal{G}_n)$ the Lie algebroids integrating to $(\mathcal{G}_n')$ and $(\mathcal{G}_n)$ respectively.
The induced map 
\begin{equation*}
    \mathcal{A}(\pi): \mathcal{A}(\mathcal{G}_n') \rightarrow \mathcal{A}(\mathcal{G}_n)
\end{equation*}
 is then an isomorphism. Since $\mathcal{A}(\mathcal{G}_n') \cong \mathcal{A}_n$ by Theorem \ref{gn'}, it follows that $\mathcal{A}(\mathcal{G}_n) \cong \mathcal{A}_n$. Therefore $\mathcal{G}_n$ integrates the Lie algebroid $\mathcal{A}_n$.
\end{proof}    
\subsection{Source fibers and action groupoid}
To describe the source fibers we need to distinguish between points on the interior of the punctured disk and on its boundary.

First consider a point $z \in \D_*$ in the interior of the disk.
We recall here the Iwasawa decomposition\footnote{Compared to Gauss decomposition, Iwasawa decomposition has the advantage of being global.} of a $\psu$ element
\begin{equation*}
    g=kan, \quad k \in K,\ a \in A,\ n\in N
\end{equation*}
where $K$ is a compact subgroup, $A$ is the subgroup of diagonal matrices with unit determinant and $N$ 
is the subgroup of upper triangular matrices with ones on the diagonal.  
\cite{Valach_2020}. Any element in $\psu$ can be factorized using this decomposition, which is unique.
Furthermore, because all elliptic subgroups in $\psu$ are conjugate to one another, one has $\psu = (AN) G_{z^n}$, where $AN \cong \D$ and $G_{z^n} = \text{Stab}_G(z^n) \cong S^1$ is the stabilizer of the point $z^n$ under Möbius transformations. 
Hence there is a unique presentation of paths $u(t) = a(t) \cdot b(t)$, where $a(t)\in AN$ and $b(t)\in G_{z^n}$. 
Modulo homotopy avoiding $0\in \D$, we obtain a description of source fibers of the Lie groupoid $\mathcal{G}_n'$:
\begin{equation*}
    \sigma'^{-1}(z) \cong \tilde{G}_{z^n} \times \tilde{\D}_* = \R \times \R \times (0, 1),
\end{equation*}
where $\tilde{\D}_* \cong \mathbb{R} \times (0,1)$ is the universal cover of the punctured disk. The equivalence relation corresponding to the path $e^{\pi nHt}$ defines an action of $\mathbb{Z}$ on $\R \times \R \times (0,1)$:
\begin{equation*}
    (\phi, \psi, r) \sim (\phi + 2\pi n, \psi + 2\pi n, r)
\end{equation*}
Therefore, we obtain two presentations of the source fiber of $\mathcal{G}_n$:
\begin{equation}\label{sfiberboundary}
    \sigma^{-1}(z) \cong \R \times \D_*^{(n)} \cong S^1 \times \tilde{\D}_*,
\end{equation}
where $\D_*^{(n)}$ is the $n$-fold cover of the punctured disk.

For a boundary point $z \in \partial \Dbarstar$, we use the decomposition
 $\psu = S^1 G_{z^n}$ to obtain
$u(t) = c(t) \cdot d(t)$ with $c(t)\in S^1$ and $d(t)\in G_{z^n}\cong AN$, the stabilizer of 
 a boundary point under Möbius. Now every homotopy avoids the origin since $u(t)\cdot z^n$ stays on the boundary. Hence, we obtain
\begin{equation*}
    \sigma'^{-1}(z) \cong \tilde{S}^1 \times \widetilde{AN} \cong \R \times AN = \widetilde{\psu}.
\end{equation*}
Factoring out paths of the form $e^{\pi n\mathcal{H}t}$, the universal cover of $\psu$ is replaced by its $n$-fold cover:
\begin{equation}\label{sfiberbulk}
    \sigma^{-1}(z) \cong \psun.
\end{equation}
We conclude that on the boundary $\mathcal{G}_n$ restricts to the action groupoid:
\begin{equation*}
    \mathcal{G}_{n} |_{\partial \Dbarstar} = \psun \times S^{1} \to S^{1}
\end{equation*}
\begin{theorem}
    The integrating groupoid $\mathcal{G}_n \rightrightarrows \Dbarstar$ is not an action Lie groupoid.
\end{theorem}

\begin{proof}
Let us consider the groupoid  $(\psu\times \Dbar)|_{\Dbarstar}$ with source map $\sigma(g,z)=z$, target map $\tau(g,z)=g\cdot z$ for all $(g,z) \in \psu \times \Dbar$, and multiplication map for any two composable arrows
    \begin{equation*}
(g_1,z)\circ(g_2,w)=(g_1\circ g_2, w).
    \end{equation*}
Consider the $n$-fold covering map
\begin{align*}
\begin{split}
    \pi:  \mathcal{G}_n &\to (\psu \times \Dbar)|_{\mathbb{D}_{*}} \\
[(u(t),z)] &\mapsto (u(1),z^n).
\end{split}
\end{align*}
Let us suppose by contradiction, that the infinitesimal action of $\rho_n$ in (\ref{actionLiealg}) integrates to global action of a Lie group $H$ on $\Dbarstar$, that is $\mathcal{G}_n \cong H\ltimes \Dbarstar$. Then the Lie algebra of $H$ is isomorphic to $\mathfrak{sl}_2(\mathbb{R})$ and there is a covering map $q:H \rightarrow \psu$. The $H$-action is compatible with the Möbius action:
\begin{equation}\label{haction}
    (h \cdot z)^n=q(h)\cdot z^n.
\end{equation}
Because $\psu$ acts transitively on $\Dbar$, there exists $g \in \psu$ such that
\begin{equation*}
    g\cdot z^n=0.
\end{equation*}
Choose a converging sequence of group elements $g_k \rightarrow g$ such that
\begin{equation}\label{sequence}
    \lim_{k \to \infty}(g_k\cdot z^n)=0  
\end{equation}
and such that
\begin{equation*}
    g_k\cdot z^n\neq 0, \quad \forall k.
\end{equation*}
 Fix a lift $\tilde{g} \in q^{-1}(g)$. There exist lifts $\tilde{g_k} \in H$ of $g_k$ such that
 \begin{equation*}
     q(\tilde{g_k})=g_k, \qquad \tilde{g_k}\rightarrow\tilde{g}.
 \end{equation*}
 Let $w_k=\tilde{g_k}\cdot z$. By equation \eqref{haction}
\begin{equation*}
w_k^n=q(\tilde{g_k})\cdot z^n=g_k \cdot z^n \rightarrow g \cdot z^n=0.
\end{equation*}
Then
\begin{equation}
w_{k}\rightarrow 0.
\end{equation}
However, continuity of $H$-action gives
\begin{equation*}
w_{k}\rightarrow \tilde{g}\cdot z \in \Dbarstar
\end{equation*}
which represents a contradiction because $0 \notin \Dbarstar$. Consequently, $\mathcal{G}_n$ cannot be an action Lie groupoid. 
\end{proof}

\begin{remark}
    Note that $(\psu \times \Dbar)|_{\Dbarstar}$ is not an action groupoid either. Proving this is more direct than the $\mathcal{G}_n$ case above. We can compute the source fibers exploiting again the Iwasawa decomposition. As before, there are two cases:
\begin{enumerate}
    \item  $z \in \text{int}(\Dbarstar)$. We obtain $\sigma^{-1}(z)=S^1 \times \D_*$.
    \item $z \in \partial(\Dbarstar)$. We obtain $\sigma^{-1}(z)=AN \times S^1 \cong \psu$.
\end{enumerate}
 Topologically these fibers are different:
 $\pi_1(S^1 \times \D_*)=\mathbb{Z}\times \mathbb{Z}$, and $\pi_1(\psu)=\mathbb{Z}$. Therefore,  $(\psu\times \mathbb{D})|_{\Dbarstar}$ cannot be written as a smooth action of a Lie group $G$ on the whole domain. Note that this argument does not work for $\mathcal{G}_n$, as the fibers in \eqref{sfiberboundary} and \eqref{sfiberbulk} have the same fundamental group.
\end{remark}

\section{A correspondence picture}\label{section5}
We present in this last section a different point of view on the problem, based on correspondences. For a fixed group element $g = \begin{pmatrix}
    a & b\\ \bar{b} & \bar{a}
\end{pmatrix}\in \psu$ the equation \eqref{nformula} defines a ramified Riemann surface as in Figure \ref{fig:surface}, consisting of $n$ sheets. 
We can embed the Riemann surface into the product $\Dbar \times \Dbar$ as the correspondence:
\begin{equation*}
\begin{tikzcd}
& S_g \arrow[dl, "\pi_1"'] \arrow[dr, "\pi_2"] & \\
\Dbar & & \Dbar
\end{tikzcd}
\end{equation*}
where $S_g$ is the surface defined by solutions $(z,w)$ of the equation
\begin{equation}
    S_g=\{(z,w)\mid w^n=g\cdot z^n\}.
\end{equation}
Composition yields
\begin{equation}
    S_{g_2} \circ S_{g_1}= \{(z,v)\mid \exists w \in \D, (z,w)\in S_{g_1} \text{ and } (w,v)\in S_{g_2}  \}\subseteq S_{g_2g_1}.
\end{equation}
In fact, one can show that in this case $S_{g_2} \circ S_{g_1}= S_{g_2g_1}$:
\begin{proof}
    \noindent$(\subseteq)$ Let $(z,v) \in S_{g_2}\circ S_{g_1}$, so there exists 
$w\in\mathbb{D}$ with $w^n = g_1\cdot z^n$ and $v^n = g_2\cdot w^n$. 
Substituting the first into the second gives
\begin{equation*}
v^n = g_2\cdot(g_1\cdot z^n) = (g_2 g_1)\cdot z^n,
\end{equation*}
where the last equality uses associativity of the $\psu$ action by 
M\"{o}bius transformations. Hence $(z,v)\in S_{g_2 g_1}$.

\medskip
\noindent$(\supseteq)$ Let $(z,v)\in S_{g_2 g_1}$, so $v^n = (g_2 g_1)\cdot z^n$. 
Define $w\in\mathbb{D}$ by
\begin{equation*}
w^n := g_1\cdot z^n.
\end{equation*}
Since $g_1\in \psu$ preserves $\mathbb{D}$, we have $g_1\cdot z^n\in\mathbb{D}$, 
so $w$ exists in $\mathbb{D}$ (any choice of $n$-th root satisfies $|w|<1$ 
since $|w^n|<1$). By construction $(z,w)\in S_{g_1}$, and
\begin{equation*}
g_2\cdot w^n = g_2\cdot(g_1\cdot z^n) = (g_2 g_1)\cdot z^n = v^n,
\end{equation*}
so $(w,v)\in S_{g_2}$. Hence $(z,v)\in S_{g_2}\circ S_{g_1}$.
\end{proof}
\begin{figure}[htp]
    \centering
    \includegraphics[width=4cm]{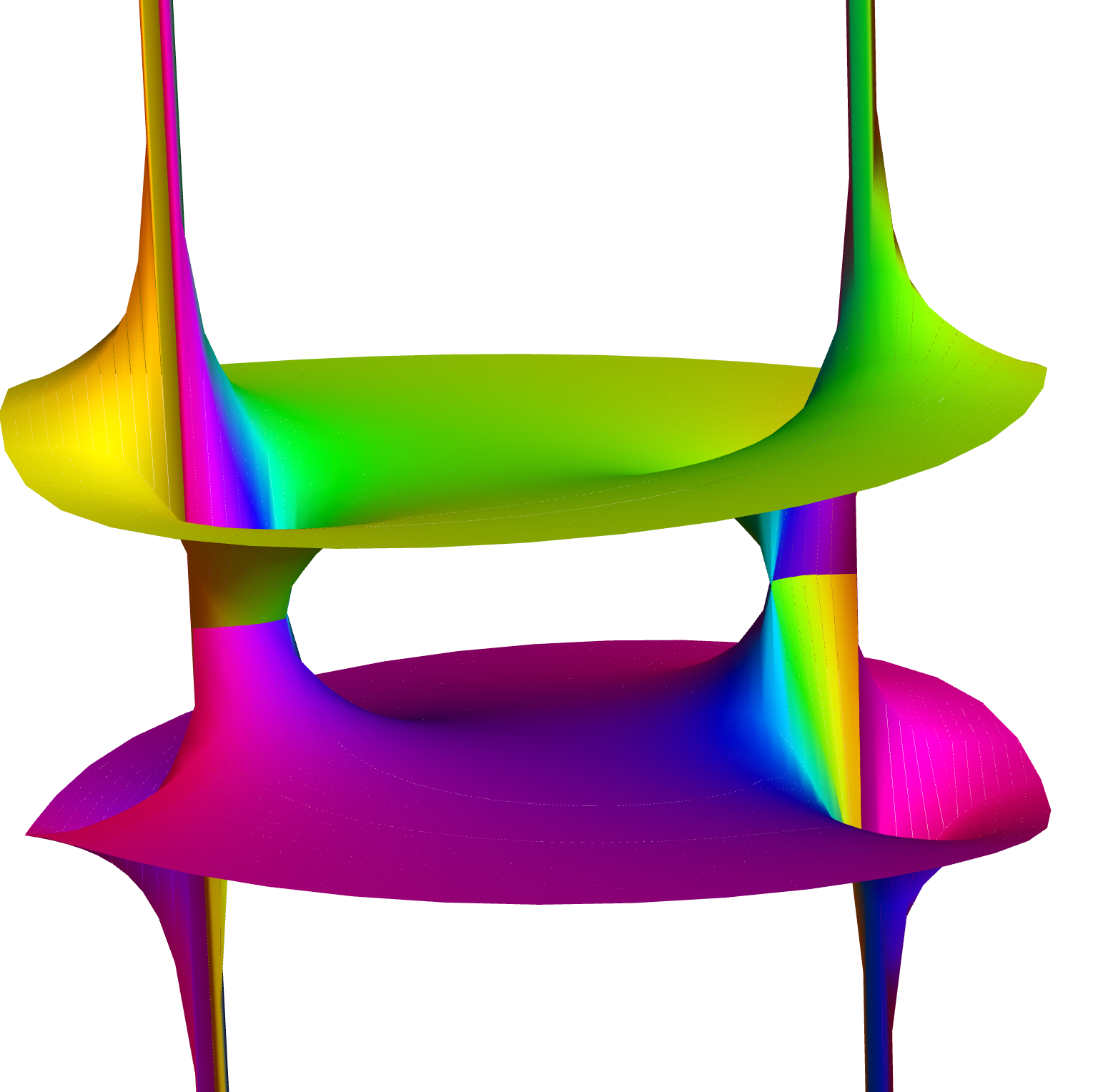}
    \caption{Projection of the ramified Riemann surface $w^2=\frac{\sqrt{2}z^2-i}{\sqrt{2}+iz^2}$.}
    \label{fig:surface}
\end{figure}

\begin{proposition}Fix $g=u(1)\in\psu$.
    There exists a forgetful map from the groupoid $\mathcal{G}_n$ to the correspondence $S_g$, sending composed arrows in the groupoid to a composition of correspondences.
\end{proposition}
\begin{proof}
    We can explicitly construct this map:
     \begin{equation}
        \begin{split}
    \phi: \mathcal{G}_n & \xrightarrow{} S_g \\
                [u,z] &\mapsto (\sigma([u,z]), \tau([u,z]))=(\pi_1(z,w), \pi_2(z,w)),\\
             \end{split}
  \end{equation}
  with $z,w \in \Dbar$.
  The map is well-defined as:
    \begin{equation}
        (\sigma([u,z]), \tau([u,z]))=(z, z(1))
    \end{equation}
    and since $z(1)^n=u(1) \cdot z^n$, then $\text{Im}(\phi)\in S_g$.
This map preserves the product structure:
\begin{equation*}
    \phi([u_1, z] \circ [u_2,w])=\phi([u_1\circ u_2, w])=(z, u_1(1)\cdot u_2(1)z)=(z, g_1g_2 \cdot z)=(\phi([u_1, z]), \phi([u_2, w])).
\end{equation*}
\end{proof}
\begin{remark}
The map in the opposite direction does not exist:  given a pair $(z,w) \in S_g$, we cannot reconstruct the path from $z$ to $w$ in the groupoid, as there exist different classes of paths $u_1, u_2$ with
\begin{equation}
    u_1(0)=u_2(0)=e \text{ and }
    u_1(1)=u_2(1)=g
\end{equation}
such that
\begin{equation}
    u_i(1)\cdot z^n=w^n, \quad i=1,2.
\end{equation}
    
\end{remark}

\printbibliography

\end{document}